\documentclass[12pt]{article}
\usepackage[a4paper,width=6.8in,height=8.8in]{geometry}
\usepackage{amsmath,amssymb}
\usepackage{amsmath,amssymb,mathtools,amsthm,
	textcomp,mathscinet}
\usepackage{amsfonts,graphicx,enumerate,subcaption,placeins}
\usepackage[hidelinks]{hyperref}
\usepackage{setspace}
\usepackage[utf8]{inputenc}
\usepackage[english]{babel}
\usepackage{booktabs,multirow}
\usepackage{blkarray}
\usepackage{lscape}
\usepackage[T1]{fontenc}
\usepackage{authblk,color}
\usepackage[english]{babel}
\usepackage{amsfonts, amstext, amsthm, booktabs, dcolumn}
\usepackage{graphicx}
\usepackage{float}
\usepackage{caption}
\usepackage{centernot}
\usepackage{times}
\usepackage[nottoc]{tocbibind}
\usepackage[mathscr]{eucal}

\definecolor{gr}{rgb}{0.7, 0.0, 0.15}

\newtheorem{theorem}{\bf Theorem}[section]

\newtheorem{corollary}{\bf Corollary}[section]
\newtheorem{remark}{\bf Remark}[section]

\newtheorem{lemma}{\bf Lemma}[section]

\newcommand\floor[1]{\left\lfloor#1\right\rfloor}
\newcommand\ceil[1]{\left\lceil#1\right\rceil}

\numberwithin{equation}{section}
\begin{document}
\title{Binomial Approximation to Locally Dependent CDO}
\author[*]{{\Large Amit N. Kumar}}
\author[$\dagger$]{{\Large P. Vellaisamy}}
\affil[*]{Department of Mathematical Sciences, Indian Institute of Technology (BHU),}
\affil[ ]{Varanasi (Uttar Pradesh) -- 221005, India.}
\affil[$\dagger$]{ Department of Mathematics, Indian Institute of Technology Bombay,}
\affil[ ]{Powai, Mumbai -- 400076, India.}
\affil[*]{Email: amit.kumar2703@gmail.com}
\affil[ $\dagger$]{Email: pv@math.iitb.ac.in}
\date{}
\maketitle

\begin{abstract}
\noindent
In this paper, we develop Stein's method for binomial approximation using the stop-loss metric that allows one to obtain a bound on the error term between the expectation of call functions. We obtain the results for a locally dependent collateralized debt obligation (CDO), under certain conditions on moments. The results are also exemplified for an independent CDO. Finally, it is shown that our bounds are sharper than the existing bounds.
\end{abstract}

\noindent

\vspace{-0.41cm}
\section{Introduction}\label{11:sec1}
The collateralized debt obligation (CDO) is a financial tool of structured asset-backed security. It is used to repackage assets into a product and sold to investors in the secondary market. These packages are in terms of debt such as auto loans, credit card debt, mortgages, and corporate debt. The assets are sliced into tranches, a set of repayment ties with different payment priorities and interest rates. The elementary tranches are senior (low risk and low return), mezzanine and equity (high risk and high return).  Investors can choose a tranche of their interest to invest. For more details, see  Jiao and Karoui \cite{JK}, Jiao {\em et al.} \cite{JKK}, Hull and White \cite{HW}, Kumar \cite{k2021, kumar2022}, Neammanee and Yonghint \cite{NY}, Yonghint {\em et al.} \cite{YNC2}, and the reference therein.\\
Condider the tranche pricing with $n$ portfolios, where each portfolio has a constant recovery rate $R>0$. Then, the percentage loss at time $T$ can be defined as (see Neammanee and Yonghint \cite[p. 2]{NY})
\vspace{-0.28cm}
\begin{align}
L(T)=\frac{(1-R)}{n}\sum_{i=1}^{n}{\bf 1}_{\{\tau_i\le T\}},\label{11:111}
\end{align} 
where $\tau_i$ is the default time of the $i$-th portfolio and ${\bf 1}_A$ denotes the indicator function of the set $A$. For a detachment point or an attachment point, say $z^*$, the main problem of CDO pricing is to evaluate the value of the mean of percentage total loss for each tranche given by (see Yonghint {\em et al.} \cite[p. 2]{YNC2})
\vspace{-0.28cm}
\begin{align*}
\mathbb{E}[(L(T)-z^*)^+]=\frac{(1-R)}{n}\mathbb{E}\left[\left(W_n-z\right)^+\right],
\end{align*} 
where $h(x)=x^+=\max\{x,0\}$ is the call function, $W_n=\sum_{i=1}^{n}{\bf 1}_{\{\tau_i\le T\}}$ and $z=({nz^*}/{(1-R)})>0$. For additional  details, see, for example, Kumar \cite{k2021, kumar2022}, Neammanee and Yonghint \cite{NY}, Yonghint {\em et al.} \cite{YNC2}.\\
In general, the mean of percentage total loss is difficult to compute when ${\bf 1}_{\{\tau_i\le T\}}$, $i=1,2,\ldots,n$, are locally dependent random variables. Therefore, it is of interest to find a suitable approximating distribution which is close to $W_n$. The proximity between $\mathbb{E}[(W_n-z)^+]$ and $\mathbb{E}[(\text{P}_\lambda-z)^+]$ is studied by Yonghint {\em et al.} \cite{YNC2}, where $\text{P}_\lambda$ denotes the Poisson random variable with parameter $\lambda$. Observe that $W_n$ is the sum of dependent Bernoulli random variables and so the binomial distribution also could a good choice  to  approximate $W_n$. See Vellaisamy and Punnen \cite{VP2001}, where it is shown that the binomial distribution may arise as the distribution of the sum of a certain dependent Bernoulli random variables. Therefore, we choose the binomial distribution as the target distribution.\\
Throughout this paper, let $\text{B}_{\alpha,p}$ follow the binomial distribution with probability mass function 
\begin{align}
\mathbb{P}(\text{B}_{\alpha,p}=k)=\binom{\alpha}{k}p^kq^{\alpha-k},\quad \text{for }k=0,1,\ldots,\alpha,\label{11:in1}
\end{align}
where $\alpha$ is a positive integer and $0<p=1-q<1$. We derive the result using the stop-loss metric defined by 
\begin{align}
d_{sl}(X,Y)= \sup_{z\in \mathbb{R}}\left|\mathbb{E}[(X-z)^+]-\mathbb{E}[(Y-z)^+]\right|,\label{11:eq15}
\end{align}
where $\mathbb{R}$ denotes the set of  real numbers. For more details, see Bautsikas and Veggelatou \cite{BV2002}.\\
The paper is organized as follows. In Section \ref{11:sec2}, we develop Stein's method for binomial distribution, under the stop-loss metric. In particular, we discuss uniform and non-uniform bounds for the solution of the Stein equation. In Section \ref{11:sec3}, we derive the error in approximation of a locally dependent CDO to a suitable binomial distribution. As a special case, we demonstrate the results for the independent CDO.  Also, we give some numerical comparisons between our results and the existing results. It is shown that our bounds improve significantly over the existing bounds.

\onehalfspacing
\section{Stein's Method}\label{11:sec2}
Stein \cite{stein1972} proposed an elegant method to find the error in approximating
the sums of real-valued random variables to a normal distribution. Then, Chen \cite{chen1975} adapted the technique for approximating the sums of discrete random variables to a 
suitable Poisson distribution. Later, numerous authors adapted or used 
 Stein's method for several distributions and under various distance metrics. To mention a few relevant ones,  Stein's method for negative binomial distribution has been studied by Brown and Phillips \cite{BP} under the total variation distance and  for the Poisson distribution has been studied by Neammanee and Yonghint \cite{NY}, under the stop-loss metric. For recent developments, see Barbour {\em et al.} \cite{BHJ}, Brown and Xia \cite{BX2001}, Eichelsbacher and Reinert \cite{ER} and Kumar \cite{kumar2022}.

\vspace*{.3cm}
\noindent In this section, we develop the Stein's method for the binomial distribution, under the stop-loss metric defined in \eqref{11:eq15}. Our work is mainly focused on finding the upper bound for $|\Delta g_z(k)|=|\ g_z(k+1)-g_z(k)|$, where $g_z(k)$ is the solution of the Stein equation given by
\begin{align}
\mathscr{A}g_z(k)=(k-z)^+-\mathbb{E}(\text{B}_{\alpha,p}-z)^+.\label{11:eq40}
\end{align}
Here, $\mathscr{A}$ denotes the Stein operator of $\text{B}_{\alpha,p}$ given by
\begin{align}
\mathscr{A}g_z(k)=\frac{(\alpha-k)p}{q}g_z(k+1)-kg_z(k), \quad \text{for }k=0,1,\ldots,\alpha.\label{11:in3}
\end{align}
Note that $\mathbb{E}(\mathscr{A} g_z(\text{B}_{\alpha,p}))=0$. Using \eqref{11:in3}, the Stein equation \eqref{11:eq40} leads to 
\begin{align*}
\frac{(\alpha-k)p}{q}g_z(k+1)-kg_z(k)=(k-z)^+-\mathbb{E}(\text{B}_{\alpha,p}-z)^+.
\end{align*}
From Section 2 of Kumar {\em et al.} \cite[p.~4]{KUV} with appropriate changes, it can  easily be verified that the solution of the above equation is
\begin{align}
g_z(k)=\left\{
\begin{array}{ll}
0,& \text{if $k=0$};\\
\displaystyle{-\sum_{j=k}^{\alpha}\frac{(\alpha-k)!}{(\alpha-j)!}\frac{(k-1)!}{j!}\left(\frac{p}{q}\right)^{j-k}[(j-z)^+-\mathbb{E}(\text{B}_{\alpha,p}-z)^+]}, & \text{if $1\le k \le \alpha$.}
\end{array}\right.\label{11:ss}
\end{align}
The bounds for $|\Delta g_z(k)|$ may be essentially uniform for all $k$ and $z$. However, the bound depends on $k$ may be useful in practice to get a sharper bound for the stop-loss metric. So, we obtain the non-uniform upper bound for $|\Delta g_z(k)|$.\\
The following lemma gives the non-uniform upper bound for $|\Delta g_z(k)|$, for all $z\ge 0$.

\begin{lemma}\label{11:smle1}
For $z\ge 0$, the following inequality holds:
\begin{align*}
|\Delta g_z (k)|\le \left\{
\begin{array}{ll}
2q^{1-\alpha}-q, & \text{for}~ k=0;\\
2q^{k-\alpha}, & \text{for}~ 1\le k \le \alpha.
\end{array}\right.
\end{align*}
\end{lemma}
\begin{proof}
It can be easily verified that $\mathbb{E}[(\text{B}_{\alpha,p}-z)^+]\le \alpha p$. For $k=0$, we have
\begin{align*}
|\Delta g_z(0)|=|g_z(1)|&\le\sum_{j=1}^{\alpha}\frac{(\alpha-1)!}{(\alpha-j)!j!}\left(\frac{p}{q}\right)^{j-1}[(j-z)^++\mathbb{E}(B_{\alpha,p}-z)^+]\\
&\le \sum_{j=1}^{\alpha}\binom{\alpha-1}{j-1}\left(\frac{p}{q}\right)^{j-1}+p\sum_{j=1}^{\alpha}\binom{\alpha}{j}\left(\frac{p}{q}\right)^{j-1}\\
&=2q^{1-\alpha}-q .
\end{align*}
For $1\le k\le \alpha-1$, let 
\begin{align*}
h_1(k)=\sum_{j=k}^{\alpha}\frac{(\alpha-k)!}{(\alpha-j)!}\frac{(k-1)!}{j!}\left(\frac{p}{q}\right)^{j-k}(j-z)^+
\end{align*}
and
\begin{align*}
h_2(k)=\sum_{j=k}^{\alpha}\frac{(\alpha-k)!}{(\alpha-j)!}\frac{(k-1)!}{j!}\left(\frac{p}{q}\right)^{j-k}\mathbb{E}(\text{B}_{\alpha,p}-z)^+.
\end{align*}
Then,
\begin{align}
\Delta g_z(k)&=g_z(k+1)-g_z(k)=[h_1(k)-h_1(k+1)]-[h_2(k)-h_2(k+1)].\label{11:eq1}
\end{align}
First, consider
\begin{align}
0<h_1(k)&=\sum_{j=k}^{\alpha}\frac{(\alpha-k)!}{(\alpha-j)!}\frac{(k-1)!}{j!}\left(\frac{p}{q}\right)^{j-k}(j-z)^+\nonumber \\
&\le \sum_{j=k}^{\alpha}\frac{(\alpha-k)!}{(\alpha-j)!}\frac{(k-1)!}{(j-1)!}\left(\frac{p}{q}\right)^{j-k}\nonumber\\
&=1+\sum_{j=k+1}^{\alpha}\frac{(\alpha-k)!}{(\alpha-j)!}\frac{1}{k(k+1)\ldots (j-1)}\left(\frac{p}{q}\right)^{j-k} \nonumber\\
&\le 1+\sum_{j=k+1}^{\alpha}\binom{\alpha-k}{j-k}\left(\frac{p}{q}\right)^{j-k}\nonumber\\
&=1+\sum_{j=1}^{\alpha-k}\binom{\alpha-k}{j}\left(\frac{p}{q}\right)^{j}\nonumber\\
&=q^{k-\alpha}.\label{11:eq2}
\end{align}
Similarly, $0<h_1(k+1)\le q^{k+1-\alpha}\le q^{k-\alpha}$. Now, consider
\begin{align}
0<h_2(k)&=\sum_{j=k}^{\alpha}\frac{(\alpha-k)!}{(\alpha-j)!}\frac{(k-1)!}{j!}\left(\frac{p}{q}\right)^{j-k}\mathbb{E}(\text{B}_{\alpha,p}-z)^+\nonumber \\
&\le \frac{\alpha p}{k} \sum_{j=k}^{\alpha}\frac{(\alpha-k)!}{(\alpha-j)!}\frac{k!}{j!}\left(\frac{p}{q}\right)^{j-k}\nonumber \\
&\le\frac{\alpha p}{k}\left(1+\sum_{j=k+1}^{\alpha}\frac{(\alpha-k)!}{(\alpha-j)!}\frac{k!}{j!}\left(\frac{p}{q}\right)^{j-k}\right)\nonumber\\
&=\frac{\alpha p}{k}\left(1+\sum_{j=k+1}^{\alpha}\frac{(\alpha-k)!}{(\alpha-j)!}\frac{1}{(k+1)\ldots j}\left(\frac{p}{q}\right)^{j-k}\right)\nonumber\\
&\le\frac{\alpha p}{k}\left(1+\frac{1}{\alpha-k+1}\sum_{j=k+1}^{\alpha}\binom{\alpha-k+1}{j-k+1}\left(\frac{p}{q}\right)^{j-k}\right)\nonumber\\
&=\frac{\alpha p}{k}\left(1+\frac{1}{\alpha-k+1}\sum_{j=1}^{\alpha-k}\binom{\alpha-k+1}{j+1}\left(\frac{p}{q}\right)^{j}\right) \nonumber\\
&=\frac{\alpha p}{k}\left(1+\frac{q^{k-\alpha}-(\alpha-k)p-1}{(\alpha-k+1)p}\right)\nonumber\\
&\le\frac{\alpha p}{k}\left(1+\frac{q^{k-\alpha}-(\alpha-k+1)p}{(\alpha-k+1)p}\right)\nonumber\\
&=\frac{\alpha pq^{k-\alpha}}{(\alpha-k+1)kp}\nonumber\\
&\le q^{k-\alpha}.\label{11:eq3}
\end{align}
Similarly, it follows that $0<h_2(k+1)\le q^{k+1-\alpha}\le q^{k-\alpha}$. Therefore, from \eqref{11:eq1}, \eqref{11:eq2} and \eqref{11:eq3}, the result follows for $1\le k \le \alpha-1$.\\
Finally, for $k=\alpha$, it is easy to verify that
\begin{align*}
|\Delta g_z(\alpha)|=|g_z(\alpha)|=\frac{(\alpha-z)^+}{\alpha}+\frac{\mathbb{E}(B_{\alpha,p}-z)^+}{\alpha}\le 2.
\end{align*}
This proves the result.
\end{proof}

\begin{remark} \em
\begin{itemize}
\item[(i)] Observe that the uniform bound can be taken as
\begin{align}
|\Delta g_z(k)|\le 2q^{1-\alpha},\quad \text{for all $z\ge 0$ and $0\le k\le \alpha$.}\label{11:eq35}
\end{align}
\item[(ii)] From \eqref{11:ss}, \eqref{11:eq2} and \eqref{11:eq3}, it can be easily seen that
\begin{align}
|g_z(k)|\le 2q^{k-n},\quad \text{for all $z\ge 0$.}\label{11:ww7}
\end{align}
\end{itemize}
\end{remark}

\noindent
Also, the uniform bound for $|g_z(k)|$ is the same as given in \eqref{11:eq35}.

\noindent
Next, using the similar technique of the proof of Lemma \ref{11:smle1}, we obtain the   following lemma which gives the non-uniform upper bound for $|\Delta g_z(k)|$ for all $z> 1$.
\begin{lemma}\label{11:smle2}
For $z>1$, the following inequality holds:
\begin{align*}
|\Delta g_z (k)|\le\left\{
\begin{array}{ll}
\vspace{0.2cm}
\displaystyle{2\left(1+\frac{q^{k-\alpha}-1}{qz}\right)}, & \text{for $k\ge z$;}\\
\vspace{0.2cm}
\displaystyle{\frac{3\left(q^{k-\alpha}-1\right)}{pz} }, & \text{for $2\le k<z$;}\\
\displaystyle{\frac{2(\alpha-1)pq^{1-\alpha}}{z} }, & \text{for $1= k<z$.}
\end{array} \right.
\end{align*}
\end{lemma}

\begin{proof}
Let $k\ge z$. First, consider
\begin{align}
0<h_1(k)&=\sum_{j=k}^{\alpha}\frac{(\alpha-k)!}{(\alpha-j)!}\frac{(k-1)!}{j!}\left(\frac{p}{q}\right)^{j-k}(j-z)^+  \nonumber\\
&\le 1+\sum_{j=k+1}^{\alpha}\frac{(\alpha-k)!}{(\alpha-j)!}\frac{(k-1)!}{(j-1)!}\left(\frac{p}{q}\right)^{j-k}\nonumber\\
&=1+\frac{(\alpha-k)p}{kq}+\frac{1}{k}\sum_{j=k+2}^{\alpha}\frac{(\alpha-k)!}{(\alpha-j)!}\frac{1}{(k+1)\ldots (j-1)}\left(\frac{p}{q}\right)^{j-k}\nonumber\\
&\le 1+\frac{(\alpha-k)p}{kq}+\frac{1}{k}\sum_{j=2}^{\alpha-k}\binom{\alpha-k}{j}\left(\frac{p}{q}\right)^{j}\nonumber\\
&=1+\frac{(\alpha-k)p}{kq}+\frac{1}{k}\left(q^{k-\alpha}-\frac{(\alpha-k)p}{q}-1\right) \nonumber\\
&\le 1+\frac{q^{k-\alpha}-1}{z}.\label{11:f1}
\end{align} 
Next, for $k\ge z$ and $k\ge \alpha p$, we have
\begin{align}
0&<\sum_{j=k}^{\alpha}\frac{(\alpha-k)!}{(\alpha-j)!}\frac{(k-1)!}{j!}\left(\frac{p}{q}\right)^{j-k}\nonumber\\
&=\frac{1}{k}\sum_{j=k}^{\alpha}\frac{(\alpha-k)!}{(\alpha-j)!}\frac{k!}{j!}\left(\frac{p}{q}\right)^{j-k}\nonumber\\
&\le \frac{1}{\alpha p}\left(1+\frac{(\alpha-k)p}{(k+1)q}+\sum_{j=k+2}^{\alpha}\frac{(\alpha-k)!}{(\alpha-j)!}\frac{k!}{j!}\left(\frac{p}{q}\right)^{j-k}\right)\nonumber\\
&= \frac{1}{\alpha p}\left(1+\frac{(\alpha-k)p}{(k+1)q}+\frac{1}{(k+1)}\sum_{j=k+2}^{\alpha}\frac{(\alpha-k)!}{(\alpha-j)!}\frac{(k+1)!}{j!}\left(\frac{p}{q}\right)^{j-k}\right)\nonumber\\
&= \frac{1}{\alpha p}\left(1+\frac{(\alpha-k)p}{(k+1)q}+\frac{1}{(k+1)}\sum_{j=k+2}^{\alpha}\frac{(\alpha-k)!}{(\alpha-j)!}\frac{1}{(k+2)\ldots j}\left(\frac{p}{q}\right)^{j-k}\right)\nonumber\\
&\le \frac{1}{\alpha p}\left(1+\frac{(\alpha-k)p}{(k+1)q}+\frac{1}{(k+1)}\sum_{j=k+2}^{\alpha}\frac{(\alpha-k)!}{(\alpha-j)!(j-k)!}\left(\frac{p}{q}\right)^{j-k}\right)\nonumber\\
&= \frac{1}{\alpha p}\left(1+\frac{(\alpha-k)p}{(k+1)q}+\frac{1}{(k+1)}\sum_{j=2}^{\alpha-k}\binom{\alpha-k}{j}\left(\frac{p}{q}\right)^{j}\right)\nonumber\\
&= \frac{1}{\alpha p}\left(1+\frac{(\alpha-k)p}{(k+1)q}+\frac{1}{(k+1)}\left(q^{k-\alpha}-\frac{(\alpha-k)p}{q}-1\right)\right)\nonumber\\
&\le \frac{1}{\alpha p}\left(1+\frac{q^{k-\alpha}-1}{z}\right).\label{11:eq6}
\end{align}
Also, for $k\ge z$ and $k\le \alpha p$, we have
\begin{align}
0&<\sum_{j=k}^{\alpha}\frac{(\alpha-k)!}{(\alpha-j)!}\frac{(k-1)!}{j!}\left(\frac{p}{q}\right)^{j-k}\nonumber\\
&=\frac{1}{\alpha-k+1}\sum_{j=k}^{\alpha}\frac{(\alpha-k+1)!}{(\alpha-j)!}\frac{(k-1)!}{j!}\left(\frac{p}{q}\right)^{j-k}\nonumber\\
&=\frac{1}{\alpha-k+1}\left(\frac{\alpha-k+1}{k}+\frac{1}{k} \sum_{j=k+1}^{\alpha}\frac{(\alpha-k+1)!}{(\alpha-j)!}\frac{k!}{j!}\left(\frac{p}{q}\right)^{j-k}\right)\nonumber\\
&= \frac{1}{\alpha-k+1}\left(\frac{\alpha-k+1}{k}+\frac{1}{k} \sum_{j=k+1}^{\alpha}\frac{(\alpha-k+1)!}{(\alpha-j)!}\frac{1}{(k+1)\ldots j}\left(\frac{p}{q}\right)^{j-k}\right)\nonumber\\
&\le \frac{1}{\alpha-k+1}\left(\frac{\alpha-k+1}{k}+\frac{1}{k} \sum_{j=k+1}^{\alpha}\frac{(\alpha-k+1)!}{(\alpha-j)!(j-k+1)!}\left(\frac{p}{q}\right)^{j-k}\right)\nonumber\\
&= \frac{1}{\alpha-k+1}\left(\frac{\alpha-k+1}{k}+\frac{1}{k} \sum_{j=1}^{\alpha}\binom{\alpha-k+1}{j+1}\left(\frac{p}{q}\right)^{j}\right)\nonumber\\
&= \frac{1}{\alpha-k+1}\left(\frac{\alpha-k+1}{k}+\frac{1}{k} \left(\frac{q^{k-\alpha}-1}{p}-(\alpha-k)\right)\right)\nonumber\\
&=\frac{1}{\alpha-k+1}\left(\frac{1}{k}+\frac{q^{k-\alpha}-1}{kp}\right)\nonumber\\
&\le \frac{1}{(\alpha-k+1)k}+\frac{q^{k-\alpha}-1}{(\alpha-k)zp}\nonumber\\
&\le \frac{1}{\alpha}+\frac{q^{k-\alpha}-1}{\alpha pqz}\nonumber\\
&\le\frac{1}{\alpha p}\left(1+\frac{q^{k-\alpha}-1}{qz}\right).\label{11:eq7}
\end{align}
Therefore, form \eqref{11:eq6} and \eqref{11:eq7}, for $k\ge z$, we have
\begin{align}
0<h_2(k)&=\sum_{j=k}^{\alpha}\frac{(\alpha-k)!}{(\alpha-j)!}\frac{(k-1)!}{j!}\left(\frac{p}{q}\right)^{j-k}\mathbb{E}(B_{\alpha,p}-z)^+ \nonumber\\
&\le 1+\frac{q^{k-\alpha}-1}{qz}.\label{11:f2}
\end{align}
Hence, from \eqref{11:eq1}, \eqref{11:f1} and \eqref{11:f2}, the result follows for $k\ge z$.\\
Next, let $2\le k <z$. Note that
\begin{align}
0<h_1(k)&=\sum_{j=\ceil{z}}^{\alpha}\frac{(\alpha-k)!}{(\alpha-j)!}\frac{(k-1)!}{j!}\left(\frac{p}{q}\right)^{j-k}(j-z)\nonumber\\
&\le\frac{1}{z} \sum_{j=\ceil{z}}^{\alpha}\frac{(\alpha-k)!}{(\alpha-j)!}\frac{(k-1)!}{(j-1)!}\left(\frac{p}{q}\right)^{j-k}(j-z)\nonumber\\
&\le\frac{1}{z} \sum_{j=k+1}^{\alpha}\frac{(\alpha-k)!}{(\alpha-j)!}\frac{(k-1)!}{(j-1)!}\left(\frac{p}{q}\right)^{j-k}(j-z)^+\nonumber\\
&\le\frac{1}{z} \sum_{j=k+1}^{\alpha}\frac{(\alpha-k)!}{(\alpha-j)!}\frac{(k-1)!}{(j-2)!}\left(\frac{p}{q}\right)^{j-k}\nonumber\\
&=\frac{1}{z} \left(\frac{(\alpha-k)p}{q}+\sum_{j=k+2}^{\alpha}\frac{(\alpha-k)!}{(\alpha-j)!}\frac{1}{k(k+1)\ldots (j-2)}\left(\frac{p}{q}\right)^{j-k}\right)\nonumber\\
&\le\frac{1}{z} \left(\frac{(\alpha-k)p}{q}+\sum_{j=k+2}^{\alpha}\frac{(\alpha-k)!}{(\alpha-j)!}\frac{1}{2.3\ldots(j-k)}\left(\frac{p}{q}\right)^{j-k}\right)\nonumber\\
&=\frac{1}{z} \left(\frac{(\alpha-k)p}{q}+\sum_{j=2}^{\alpha-k}\binom{\alpha-k}{j}\left(\frac{p}{q}\right)^{j}\right)\nonumber\\
&=\frac{q^{k-\alpha}-1}{z}.\label{11:eq4}
\end{align}
Next, consider
\begin{align*}
0&<\sum_{j=k}^{\alpha}\frac{(\alpha-k)!}{(\alpha-j)!}\frac{(k-1)!}{j!}\left(\frac{p}{q}\right)^{j-k}\\
&=\frac{1}{k}\sum_{j=k}^{\alpha}\frac{(\alpha-k)!}{(\alpha-j)!}\frac{k!}{j!}\left(\frac{p}{q}\right)^{j-k}\\
&=\frac{1}{k}\left(1+\sum_{j=k+1}^{\alpha}\frac{(\alpha-k)!}{(\alpha-j)!}\frac{1}{(k+1)\ldots j}\left(\frac{p}{q}\right)^{j-k}\right)\\
&\le\frac{1}{k}\left(1+2\sum_{j=k+1}^{\alpha}\frac{(\alpha-k)!}{(\alpha-j)!(j-k+2)!}\left(\frac{p}{q}\right)^{j-k}\right)\\
&\le\frac{1}{k}\left(1+\frac{2}{(\alpha-k)(\alpha-k+1)}\sum_{j=1}^{\alpha-k}\binom{\alpha-k+2}{j+2}\left(\frac{p}{q}\right)^{j}\right)\\
&=\frac{1}{k}\left(1+\frac{2}{(\alpha-k)(\alpha-k+1)}\left(\frac{2q^{k-\alpha}-2-2(\alpha-k)p-(\alpha-k)(\alpha-k+1)p^2}{2p^2}\right)\right)\\
&\le \frac{2(q^{k-\alpha}-1)}{(\alpha-k)(\alpha-k+1)kp^2}\\
&\le  \frac{2(q^{k-\alpha}-1)}{\alpha(\alpha-z)p^2}.
\end{align*}
Note that, for $k<z\le \alpha$, we have
\begin{align*}
\mathbb{E}(B_{\alpha,p}-z)^+&=\sum_{m=\ceil{z}}^{\alpha}(m-z)\binom{\alpha}{m}p^mq^{\alpha-m}\\
&\le (\alpha-z)\sum_{m=\ceil{z}}^{\alpha}\frac{\alpha!}{(\alpha-m!)m!}p^mq^{\alpha-m}\\
&\le \frac{\alpha(\alpha-z)}{z}\sum_{m=1}^{\alpha}\binom{\alpha-1}{m-1}p^mq^{\alpha-m}\\
&=\frac{\alpha(\alpha-z)p}{z}
\end{align*}
Therefore,
\begin{align}
0<h_2(k)&=\sum_{j=k}^{\alpha}\frac{(\alpha-k)!}{(\alpha-j)!}\frac{(k-1)!}{j!}\left(\frac{p}{q}\right)^{j-k}\mathbb{E}(B_{\alpha,p}-z)^+\nonumber\\
&\le \frac{2(q^{k-\alpha}-1)}{pz}.\label{11:eq5}
\end{align}
Hence, from \eqref{11:eq1}, \eqref{11:eq4} and \eqref{11:eq5}, the result follows for $2\le k<z \le \alpha$.\\
Next, for $1=k<z$, we have
\begin{align*}
0<h_1(1)&=\sum_{j=\ceil{z}}^{\alpha}\frac{(\alpha-1)!}{(\alpha-j)!j!}\left(\frac{p}{q}\right)^{j-1}(j-z)\\
&\le \frac{\alpha-1}{z}\sum_{j=2}^{\alpha}\binom{\alpha-2}{j-2}\left(\frac{p}{q}\right)^{j-1}\\
&=\frac{(\alpha-1)pq^{1-\alpha}}{z}.
\end{align*}
and 
\begin{align*}
0<h_2(1)&=\sum_{j=1}^{\alpha}\frac{(\alpha-1)!}{(\alpha-j)!j!}\left(\frac{p}{q}\right)^{j-1}\mathbb{E}(B_{\alpha,p}-z)^+\\
&\le \left(\frac{q^{1-\alpha}-q}{\alpha p}\right)\left(\frac{\alpha(\alpha-1)p^2}{z}\right)\\
&=\frac{(\alpha-1)pq^{1-\alpha}}{z}.
\end{align*}
This proves the result.
\end{proof}

\section{Bounds for Binomial Approximation}\label{11:sec3}
In this section, we obtain the error bounds for binomial approximation to locally dependent CDO. We derive the results for the stop-loss metric under certain conditions on moments. Moreover, we demonstrate the results under an independent setup. It is shown that binomial distribution is more suitable for a CDO using the numerical comparison between our bounds and the existing bound given by Neammanee and Yonghint \cite{NY}.\\
In Yonghint {\em et al.} \cite{YNC2}, it is shown that the locally dependent CDO is useful in real-life applications in various aspects.  So, we consider a similar locally dependent structure that can also be used for independent setup. Let $X_1,X_2,\ldots,X_n$ be a collection of random variables such that $X_i$ is independent of $X_{A_i^c}$, while $X_{A_i}$ is independent of $X_{B_i^c}$, where $i\in A_i\subseteq B_i\subset \{1,2,\ldots,n\}$, $i=1,2,\ldots,n$. Here, $X_A$ denotes the collection of random variables $\{X_i,i \in A\}$, and $A^c$ denotes the complement of the set $A$. Note that if $A_i=B_i=\{i\}$, then $X_1,X_2,\ldots,X_n$ become independent random variables. See Kumar \cite{k2021, kumar2022}, R\"{o}llin \cite{RO2008} and \v{C}ekanavi\v{c}ius and Vellaisamy \cite{CV2015,CV2021} for a similar locally dependent setup.  

\noindent Let $\tau_i$ is the default time of the $i$-th portfolio  and $X_i$ henceforth denote the random variable ${\bf 1}_{\{\tau_i\le T\}}$ with $\mathbb{P}(X_i=1)=p_i=1-q_i=1-\mathbb{P}(X_i=0)$. Consider
\begin{align}
W_n=\sum_{i=1}^{n}X_i, \label{11:in2}
\end{align}
which is a key factor of percentage loss up to time $T$ defined in \eqref{11:111}. Our aim is to approximate $W_n$ by a suitable binomial random variable.

\noindent Throughout this section, let $g_z=g$,
\begin{align}
W_i=W_n-X_{A_i}=\sum_{j\notin A_i}X_i\quad \text{and}\quad W_{i}^{*}=W_n-X_{B_i}=\sum_{j\notin B_i}X_i.\label{11:opa2}
\end{align}
Note that $X_i$ is independent of $W_i$. Also, $X_i$ and $X_{A_i}$ are independent of $W_i^{*}$. \\
First, we choose one parameter of binomial distribution of our choice, and accordingly, another parameter can be obtained. Since $\alpha$ should be a positive integer (number of identical Bernoulli trials) for the binomial distribution, we choose $\alpha = n$ and let 
\begin{align}
p=\frac{1}{\alpha}\sum_{i=1}^{n}\mathbb{E}(X_i)=\frac{1}{n}\sum_{i=1}^{n}p_i.\label{11:opa1}
\end{align}
The following theorem gives the error in the approximation between $\text{B}_{\alpha,p}$ and $W_n$.

\begin{theorem}\label{11:opath1}
Let $\text{B}_{\alpha,p}$ and $W_n$ be defined as in \eqref{11:in1} and \eqref{11:in2}, respectively, and satisfy \eqref{11:opa1}. Then
\begin{align}
d_{sl}(W_n,\text{B}_{n,p})\le \frac{2}{pq^n}\sum_{i=1}^{n}\left[\mathbb{E}\left((X_i+p)q^{W_i}\right)-\mathbb{E}\left((p_i+qX_i)q^{W_n}\right)\right].\label{11:ww1}
\end{align}
where $d_{sl}$ denote the stop-loss distance defined in \eqref{11:eq15}.
\end{theorem}

\begin{proof}
Multiplying the Stein operator \eqref{11:in3} by $q$ and taking expectation with respect to $W_n$, we have
\begin{align}
q\mathbb{E}[\mathscr{A}g(W_n)]&=\alpha p\mathbb{E}(g(W_n+1))-p\mathbb{E}(W_ng(W_n+1))-q\mathbb{E}(W_ng(W_n)).\label{11:eq61}
\end{align}
Using \eqref{11:opa1}, we get
\begin{align*}
q\mathbb{E}[\mathscr{A}g(W_n)]&=\sum_{i=1}^{n}\mathbb{E}(X_i)\mathbb{E}(g(W_n+1))-\sum_{i=1}^{n}\mathbb{E}(X_ig(W_n+1))+q\sum_{i=1}^{n}\mathbb{E}(X_i \Delta g(W_n))
\end{align*}
From \eqref{11:opa2}, note that $X_i$ and $W_i$ are independent. Therefore, the above expression can be written as
\begin{align}
q\mathbb{E}[\mathscr{A}g(W_n)]&=\sum_{i=1}^{n}\mathbb{E}(X_i)\mathbb{E}(g(W_n+1)-g(W_i+1))+q\sum_{i=1}^{n}\mathbb{E}(X_i \Delta g(W_n))\nonumber\\
&~~~-\sum_{i=1}^{n}\mathbb{E}(X_i(g(W_n+1)-g(W_i+1)))\nonumber\\
&=\sum_{i=1}^{n}\mathbb{E}(X_i)\mathbb{E}(g(W_i+X_{A_i}+1)-g(W_i+1))+q\sum_{i=1}^{n}\mathbb{E}(X_i \Delta g(W_n))\nonumber\\
&~~~-\sum_{i=1}^{n}\mathbb{E}(X_i(g(W_i+X_{A_i}+1)-g(W_i+1)))\nonumber\\
&=\sum_{i=1}^{n}\mathbb{E}(X_i)\mathbb{E}\left(\sum_{j=1}^{X_{A_i}}\Delta g(W_i+j)\right)+q\sum_{i=1}^{n}\mathbb{E}(X_i \Delta g(W_n))\nonumber\\
&~~~-\sum_{i=1}^{n}\mathbb{E}\left(X_i\sum_{j=1}^{X_{A_i}}\Delta g(W_i+j)\right).\label{11:ww5}
\end{align}
Using Lemma \ref{11:smle1}, we get
\begin{align*}
|\mathbb{E}[\mathscr{A}g(W_n)]|&\le \frac{2}{pq^n}\sum_{i=1}^{n}p_i\mathbb{E}\left(q^{W_i}(1-q^{X_{A_i}})\right)+\frac{2}{q^n}\sum_{i=1}^{n}\mathbb{E}\left(X_i q^{W_n}\right)\\
&~~~+\frac{2}{pq^n}\sum_{i=1}^{n}\mathbb{E}\left(X_iq^{W_i}\left(1-q^{X_{A_i}}\right)\right)\\
&=\frac{2}{pq^n}\sum_{i=1}^{n}\left[\mathbb{E}\left((X_i+p)q^{W_i}\right)-\mathbb{E}\left((p_i+qX_i)q^{W_n}\right)\right].
\end{align*}
This proves the result.
\end{proof}

\begin{corollary}\label{11:cor1}
Let $\text{B}_{\alpha,p}$ and $W_n$ be defined as in \eqref{11:in1} and \eqref{11:in2}, respectively, and satisfy \eqref{11:opa1}. Assume $X_1,X_2,\ldots,X_n$ are independent random variables then
\begin{align}
d_{sl}(W_n,\text{B}_{n,p})\le \frac{2}{q^n}\sum_{i=1}^{n}|p-p_i|p_i \prod_{j\neq i}(1-pp_j).\label{11:eq50}
\end{align}
\end{corollary}
\begin{proof}
Substituting $A_i=\{i\}$ in \eqref{11:ww5}, it can be easily verified that
\begin{align*}
q\mathbb{E}[\mathscr{A}g(W_n)]&=\sum_{i=1}^{n}p_i^2\mathbb{E}(\Delta g(W_i+1))+q\sum_{i=1}^{n}p_i\mathbb{E}(\Delta g(W_i+1))-\sum_{i=1}^{n}p_i\mathbb{E}(\Delta g(W_i+1))\\
&=\sum_{i=1}^{n}(p_i-p)p_i\mathbb{E}(\Delta g(W_i+1)).
\end{align*}
Hence, using Lemma \ref{11:smle1}, we get
\begin{align*}
|\mathbb{E}[\mathscr{A}g(W_n)]|&\le \frac{2}{q^n}\sum_{i=1}^{n}|p-p_i|p_i\mathbb{E}\left(q^{W_i}\right)=\frac{2}{q^n}\sum_{i=1}^{n}|p-p_i|p_i \prod_{j\neq i}(1-pp_j).
\end{align*}
This proves the result.
\end{proof}

\begin{remark} \em
\begin{enumerate}
\item[(i)] Note that if $p_i=p$, $i=1,2,\ldots,n$, in \eqref{11:eq50} then $d_{sl}(W_n,\text{B}_{n,p})=0$, as expected.
\item[(ii)] In Theorem \ref{11:opath1} and Corollary \ref{11:cor1}, the bounds become shaper for sufficiently small values of $p_i$.
\item[(iii)] If we choose the parameter $p$ of our choice, then $\alpha=\frac{1}{p}\sum_{i=1}^{n}\mathbb{E}(X_i)$. In this situation, the parameter $\alpha$ may not be an integer. So, we can take 
\begin{align*}
\alpha=\floor{\frac{1}{p}\sum_{i=1}^{n}\mathbb{E}(X_i)},
\end{align*} 
where $\floor{x}$ is the  integer part of $x>0$, and use \eqref{11:ww7} and \eqref{11:eq70} to get
\begin{align}
d_{sl}(W_n,\text{B}_{\alpha,p})\le\frac{2}{pq^\alpha}\sum_{i=1}^{n}\left[\mathbb{E}\left((X_i+p)q^{W_i}\right)-\mathbb{E}\left((p_i+qX_i-\delta p^2)q^{W_n}\right)\right].\label{11:ww6}
\end{align}
Therefore, it is suggested to consider the minimum of the bounds given in \eqref{11:ww1} and \eqref{11:ww6}.
\end{enumerate}
\end{remark}

\noindent
If we choose $\alpha p=\mathbb{E}(W_n)$ and $\alpha p q=\mathrm{Var}(W_n)$ (the first two moments matching condition), then the choice of $\alpha$ may not be a positive integer. So, we choose
\begin{align}
\alpha=\floor{\frac{(\mathbb{E}(W_n))^2}{\mathbb{E}(W_n)-\mathrm{Var}(W_n)}}\quad \text{and}\quad p=\frac{\mathbb{E}(W_n)-\mathrm{Var}(W_n)}{\mathbb{E}(W_n)}.\label{11:opa3}
\end{align}
Also, define $D(Z):=2d_{TV}(Z,Z+1)$ and
\begin{align*}
\delta:=\frac{(\mathbb{E}(W_n))^2}{\mathbb{E}(W_n)-\mathrm{Var}(W_n)}-\alpha.
\end{align*}
Observe that $0\le \delta<1$,
\begin{align}
\alpha p=\mathbb{E}(W_n)-\delta p\quad \text{and}\quad \alpha p q=\mathrm{Var}(W_n)-\delta pq\label{11:eq62}
\end{align}
The next result gives the error in approximation between $\text{B}_{\alpha,p}$ and $W_n$ satisfying the above conditions.
\begin{theorem}\label{11:opath2}
Let $\text{B}_{\alpha,p}$ and $W_n$ be as defined in \eqref{11:in1} and \eqref{11:in2}, respectively, and satisfy \eqref{11:opa3}. Then
\begin{align}
d_{sl}(W_n,\text{B}_{\alpha,p})&\le \frac{2}{p^2q^{\alpha}}\left\{\sum_{i=1}^{n}\mathbb{E}(X_i)\mathbb{E}\left(\left(pX_{A_i}+q^{X_{B_i}}\left(1-q^{-X_{A_i}}\right)\right)D(W_i^*|X_{A_i},X_{B_i})\right)\right.\nonumber\\
&~~~+\sum_{i=1}^{n}\mathbb{E}\left(X_i\left(pX_{A_i}+q^{X_{B_i}}\left(1-q^{-X_{A_i}}\right)\right)D(W_i^*|X_i,X_{A_i},X_{B_i})\right)\nonumber\\
&~~~+p\sum_{i=1}^{n}\mathbb{E}\left(X_i\left(q^{B_i}-q\right)D(W_i^*|X_{B_i})\right)+\frac{\delta p^3}{q}\mathbb{E}(q^{W_n})\nonumber\\
&~~~\left.+\frac{p}{q}\sum_{i=1}^{n}|\mathbb{E}(X_i)\mathbb{E}(X_{A_i})-\mathbb{E}(X_iX_{A_i})+q\mathbb{E}(X_i)|\mathbb{E}\left(\left(q^{B_i}-q\right)D(W_i^*|X_{B_i})\right)\right\}.\label{11:eq11}
\end{align}
\end{theorem}

\begin{proof}
From \eqref{11:eq61}, we have
\begin{align*}
q\mathbb{E}[\mathscr{A}g(W_n)]&=\alpha p\mathbb{E}(g(W_n+1))-p\mathbb{E}(W_ng(W_n+1))-q\mathbb{E}(W_ng(W_n)).
\end{align*}
Using \eqref{11:eq62} and following the steps similar to the proof of Theorem \ref{11:opath1}, we get
\begin{align}
q\mathbb{E}[\mathscr{A}g(W_n)]&=\sum_{i=1}^{n}\mathbb{E}(X_i)\mathbb{E}\left(\sum_{j=1}^{X_{A_i}}\Delta g(W_i+j)\right)-\sum_{i=1}^{n}\mathbb{E}\left(X_i\sum_{j=1}^{X_{A_i}}\Delta g(W_i+j)\right)\nonumber\\
&~~~+q\sum_{i=1}^{n}\mathbb{E}(X_i \Delta g(W_n))-\delta p\mathbb{E}(g(W_n+1)).\label{11:eq70}
\end{align}
Using \eqref{11:opa2} and \eqref{11:eq62}, the above expression leads to
\begin{align}
q\mathbb{E}[\mathscr{A}g(W_n)]&=\sum_{i=1}^{n}\mathbb{E}(X_i)\mathbb{E}\left(\sum_{j=1}^{X_{A_i}}(\Delta g(W_i+j)-\Delta g(W_i^*+1))\right)\nonumber\\
&~~~-\sum_{i=1}^{n}\mathbb{E}\left(X_i\sum_{j=1}^{X_{A_i}}(\Delta g(W_i+j)-\Delta g(W_i^*+1))\right)\nonumber\\
&~~~+q\sum_{i=1}^{n}\mathbb{E}(X_i (\Delta g(W_n)-\Delta g(W_i^*+1)))-\delta p\mathbb{E}(g(W_n+1))\nonumber\\
&~~~-\sum_{i=1}^{n}[\mathbb{E}(X_i)\mathbb{E}(X_{A_i})-\mathbb{E}(X_iX_{A_i})+q\mathbb{E}(X_i)]\mathbb{E}(g(W_n+1)-g(W_i^*+1))\nonumber\\
&=\sum_{i=1}^{n}\mathbb{E}(X_i)\mathbb{E}\left(\sum_{j=1}^{X_{A_i}}\sum_{\ell=1}^{X_{B_i\backslash A_i}+j-1}\Delta^2 g(W_i^*+\ell)\right)\nonumber\\
&~~~-\sum_{i=1}^{n}\mathbb{E}\left(X_i\sum_{j=1}^{X_{A_i}}\sum_{\ell=1}^{X_{B_i\backslash A_i}+j-1}\Delta^2 g(W_i^*+\ell)\right)\nonumber\\
&~~~+q\sum_{i=1}^{n}\mathbb{E}\left(X_i \sum_{\ell=1}^{X_{B_i}-1}\Delta^2 g(W_i^*+\ell)\right)-\delta p\mathbb{E}(g(W_n+1))\nonumber\\
&~~~-\sum_{i=1}^{n}[\mathbb{E}(X_i)\mathbb{E}(X_{A_i})-\mathbb{E}(X_iX_{A_i})+q\mathbb{E}(X_i)]\mathbb{E}\left(\sum_{\ell=1}^{X_{B_i}-1}\Delta^2 g(W_i^*+\ell)\right)\label{11:eq10}\\
&=\sum_{i=1}^{n}\mathbb{E}(X_i)\mathbb{E}\left(\sum_{j=1}^{X_{A_i}}\sum_{\ell=1}^{X_{B_i\backslash A_i}+j-1}\mathbb{E}(\Delta^2 g(W_i^*+\ell)|X_{A_i},X_{B_i})\right)\nonumber\\
&~~~-\sum_{i=1}^{n}\mathbb{E}\left(X_i\sum_{j=1}^{X_{A_i}}\sum_{\ell=1}^{X_{B_i\backslash A_i}+j-1}\mathbb{E}(\Delta^2 g(W_i^*+\ell)|X_i,X_{A_i},X_{B_i})\right)\nonumber\\
&~~~+q\sum_{i=1}^{n}\mathbb{E}\left(X_i \sum_{\ell=1}^{X_{B_i}-1}\mathbb{E}(\Delta^2 g(W_i^*+\ell)|X_{B_i})\right)-\delta p\mathbb{E}(g(W_n+1))\nonumber\\
&~~~-\sum_{i=1}^{n}[\mathbb{E}(X_i)\mathbb{E}(X_{A_i})-\mathbb{E}(X_iX_{A_i})+q\mathbb{E}(X_i)]\mathbb{E}\left(\sum_{\ell=1}^{X_{B_i}-1}\mathbb{E}(\Delta^2 g(W_i^*+\ell)|X_{B_i})\right)\nonumber
\end{align}
Note that $\mathbb{E}(\Delta^2 g(W_i^*+\ell)|\cdot)\le 2q^{\ell-\alpha}D(W_i^*|\cdot)$. Hence, using \eqref{11:ww7}, we get
\begin{align*}
|\mathbb{E}[\mathscr{A}g(W_n)]|&\le \frac{2}{p^2q^{\alpha}}\left\{\sum_{i=1}^{n}\mathbb{E}(X_i)\mathbb{E}\left(\left(pX_{A_i}+q^{X_{B_i}}\left(1-q^{-X_{A_i}}\right)\right)D(W_i^*|X_{A_i},X_{B_i})\right)\right.\\
&~~~+\sum_{i=1}^{n}\mathbb{E}\left(X_i\left(pX_{A_i}+q^{X_{B_i}}\left(1-q^{-X_{A_i}}\right)\right)D(W_i^*|X_i,X_{A_i},X_{B_i})\right)\\
&~~~+p\sum_{i=1}^{n}\mathbb{E}\left(X_i\left(q^{B_i}-q\right)D(W_i^*|X_{B_i})\right)+\frac{\delta p^3}{q}\mathbb{E}(q^{W_n})\\
&~~~\left.+\frac{p}{q}\sum_{i=1}^{n}|\mathbb{E}(X_i)\mathbb{E}(X_{A_i})-\mathbb{E}(X_iX_{A_i})+q\mathbb{E}(X_i)|\mathbb{E}\left(\left(q^{B_i}-q\right)D(W_i^*|X_{B_i})\right)\right\}.
\end{align*}
This proves the result.
\end{proof}

\begin{corollary}
Let $\text{B}_{\alpha,p}$ and $W_n$ be defined as in \eqref{11:in1} and \eqref{11:opa3}, respectively, and satisfy \eqref{11:opa2}. If $X_1,X_2,\ldots,X_n$ are independent random variables, then
\begin{align}
d_{sl}(W_n,\text{B}_{\alpha,p})\le \frac{2}{q^\alpha}\left\{\sqrt{\frac{2}{\pi}}\left(\frac{1}{4}+\sum_{i=1}^{n}\gamma_i-\gamma^*\right)^{-1/2}\sum_{i=1}^{n}|p-p_i|p_i^2+\delta p\prod_{i=1}^{n}(1-pp_i)\right\},\label{11:ww10}
\end{align}
where $\gamma_j=\min\{\frac{1}{2},1-\frac{1}{2}(q_j+|q_j-p_j|)\}$ and $\gamma^*=\max_{1\le j \le n}\gamma_j$.
\end{corollary}
\begin{proof}
Substituting $A_i=\{i\}$ in \eqref{11:eq70}, it can be easily verified that
\begin{align*}
q\mathbb{E}[\mathscr{A}g(W_n)]&=\sum_{i=1}^{n}(p_i-p)p_i\mathbb{E}(\Delta g(W_i+1))-\delta p \mathbb{E}(g(W_n+1)).
\end{align*}
Using \eqref{11:eq62}, we get
\begin{align*}
q\mathbb{E}[\mathscr{A}g(W_n)]&=-\sum_{i=1}^{n}(p-p_i)p_i\mathbb{E}(\Delta g(W_n+1)-\Delta g(W_i+1))-\delta p \mathbb{E}(g(W_n+1))\\
&=-\sum_{i=1}^{n}(p-p_i)p_i^2\mathbb{E}(\Delta^2 g(W_i+1))-\delta p \mathbb{E}(g(W_n+1)
\end{align*}
Note that $|\mathbb{E}(\Delta^2 g(W_i+1))|\le {2\gamma}/{q^{\alpha-1}}$, where $\gamma=2\max_{i\in J}d_{TV}(W_i,W_i+1)$ (see Barbour and Xia \cite{BX}, and Barbour and \v{C}ekanavi\v{c}ius \cite[p. 517]{BC})). Also, from Corollary 1.6 of Mattner and Roos \cite{MR} (see also Remark 4.1 of Vellaisamy {\em et al.} \cite{VUC}), we have 
\begin{align*}
\gamma\le \sqrt{\frac{2}{\pi}}\left(\frac{1}{4}+\sum_{j=1}^{n}\gamma_j-\gamma^*\right)^{-1/2},
\end{align*} 
where
\vspace{-0.18cm}
\begin{align*}
\gamma_j&=\min\left\{\frac{1}{2},1-d_{TV}(X_j,X_j+1)\right\}\\
&=\min\left\{\frac{1}{2},1-\frac{1}{2}(q_j+|q_j-p_j|)\right\}
\end{align*}
and $\gamma^*=\max_{1\le j \le n}\gamma_j$. Hence,
\begin{align*}
|\mathbb{E}[\mathscr{A}g(W_n)]|&\le \frac{2}{q^\alpha}\left\{\sqrt{\frac{2}{\pi}}\left(\frac{1}{4}+\sum_{i=1}^{n}\gamma_i-\gamma^*\right)^{-1/2}\sum_{i=1}^{n}|p-p_i|p_i^2+\delta p\prod_{i=1}^{n}(1-pp_i)\right\}.
\end{align*}
This proves the result.
\end{proof}

\begin{remark} \em
\begin{itemize}
\item[(i)] Note that $W_i^*$ can be expressed as the conditional sum of independent random variables. Therefore, Subsections 5.3 and 5.4 of R\"{o}llin \cite{RO2008}  and Remark 3.1(ii) of Kumar {\em et al.} \cite{KUV} are useful to find the upper bound of $D(W_i^* |\cdot)$.
\item[(ii)] Observe that $|\Delta^2 g(X+\ell)|\le 4q^{X+\ell-\alpha}$. Therefore, from \eqref{11:eq10}, we get 
\begin{align}
d_{sl}(W_n,\text{B}_{\alpha,p})&\le \frac{4}{p^2q^{\alpha}}\left\{\sum_{i=1}^{n}\mathbb{E}\left((X_i+p_i)\left(pX_{A_i}q^{W_i^*}+q^{W_n}-q^{W_i}\right)\right)\right.\nonumber\\
&~~~+p\sum_{i=1}^{n}\mathbb{E}\left(X_i\left(q^{W_n}-q^{W_i^*+1}\right)\right)+\frac{\delta p^3}{2}\mathbb{E}\left(q^{W_n}\right)\nonumber\\
&~~~\left.+\frac{p}{q}\sum_{i=1}^{n}|\mathbb{E}(X_i)\mathbb{E}(X_{A_i})-\mathbb{E}(X_iX_{A_i})+q\mathbb{E}(X_i)|\mathbb{E}\left(q^{W_n}-q^{W_i^*+1}\right)\right\}.\label{11:eq12}
\end{align}
Therefore,  in practice, one could  take the minimum of the bounds \eqref{11:eq11} and \eqref{11:eq12}.
\item[(iii)]  Theorems \ref{11:opath1} and \ref{11:opath2} are established using  Lemma \ref{11:smle1} for all $z\ge 0$.  Following the steps similar to the proofs of Theorems \ref{11:opath1} and \ref{11:opath2}, the results can also be derived using Lemma \ref{11:smle2} in terms of $z>1$. This can be used to approximate $\mathbb{E}[(W_n-z)^+]$ by $\mathbb{E}[(\text{B}_{\alpha,p}-z)^+]$ for $z>1$.
\item[(iv)] From Corollary 1 of Neammanee and Yonghint \cite{NY}, we have
\begin{align}
d_{sl}(\text{P}_\lambda,W_n)\le (2e^{\lambda}-1)\sum_{i=1}^{n}p_i^2,\label{11:ny1}
\end{align}
where $\lambda=\sum_{i=1}^{n}p_i$. The bound given in \eqref{11:eq50} and \eqref{11:ww10} are better than the above bound, for example, let $n=100$ and $p_i$, $1\le i\le 100$, be defined as follows:
 \begin{table}[H]
  \centering
  \caption{The values of $p_i$}
  \begin{tabular}{cccccccccccc}
\toprule
$i$ & $p_i$ & $i$ & $p_i$ & $i$ & $p_i$ & $i$ & $p_i$ & $i$ & $p_i$\\
\midrule
1-20 & 0.06 &21-40 & 0.07 & 41-60 & 0.08 & 61-80 & 0.09 &  81-100 & 0.10  \\
\bottomrule
\end{tabular}
\end{table}
Then, the following table gives a comparison between our bounds \eqref{11:eq50} and \eqref{11:ww10}, and the existing bound \eqref{11:ny1}.
\begin{table}[H]  
  \centering
  \caption{Comparison of bounds.}
  \begin{tabular}{cccc}
\toprule
$n$& From \eqref{11:ny1} (existing bound)  & From \eqref{11:eq50} & From \eqref{11:ww10} \\
\midrule
$10$ & $0.095193$ & $0$ & $7.6 \times 10^{-16}$\\ 
$20$ & $0.406097$ & $0$ & $6.8 \times 10^{-15}$\\
$30$ & $1.496990$ & $0.109842$ & $0.638717$\\
$40$ & $4.407670$ & $0.324195$ & $1.188300$\\
$50$ & $13.78920$ & $1.186000$ & $1.474570$\\
$60$ & $39.44710$ & $3.261280$ & $1.676520$\\
$70$ & $123.9500$ & $12.78810$ & $12.56050$\\
$80$ & $370.6940$ & $39.29820$ & $13.90400$\\
$90$ & $1227.670$ & $136.3000$ & $68.75740$\\
$100$&$3934.200$ & $425.1760$ & $335.1310$\\         
\bottomrule
\end{tabular}
\end{table}
For $1\le n\le 20$, note that the bounds given in \eqref{11:eq50} are zero, as expected. Further, our bounds improve upon  the existing bounds for various values of $p_i$. Also, for sufficiently large values of $n$, the bound given in \eqref{11:ww10} is better than the bound given in \eqref{11:eq50}. 
\end{itemize}
\end{remark}

\setstretch{1.11}
\bibliographystyle{PV}
\bibliography{PA2PSD}

\end{document}